\documentclass[a4]{amsart}
\usepackage{hyperref}
\usepackage{cleveref}
\usepackage[margin=3.5cm]{geometry}
\usepackage{xcolor}

\makeatletter
\@namedef{subjclassname@2020}{\textup{2020} Mathematics Subject Classification}
\makeatother

\allowdisplaybreaks

\theoremstyle{cupthm}
\newtheorem{thm}{Theorem}[section]
\newtheorem{prop}[thm]{Proposition}
\newtheorem{cor}[thm]{Corollary}
\newtheorem{lemma}[thm]{Lemma}
\theoremstyle{cupdefn}

\theoremstyle{cuprem}
\newtheorem{rem}[thm]{Remark}

\numberwithin{equation}{section}

\newcommand{\ot}{\otimes}
\newcommand{\obp}{\otimes^{\gamma}}

\newcommand{\N}{\mathbb N}

\newcommand{\C}{\mathbb C}
\newcommand{\rar}{\rightarrow}
\newcommand{\mH}{\mathcal{H}}
\newcommand{\mK}{\mathcal{K}}
\newcommand{\mU}{\mathcal{U}}

\def\la{\langle}
\def\ra{\rangle}

\begin{document}
\title[On SAI of projective tensor product of Hilbert-Schmidt
  space]{On strong Arens irregularity of projective tensor product of
  Hilbert-Schmidt space}

\author[V P Gupta]{Ved Prakash Gupta}

\author[L K Singh]{Lav Kumar Singh*}

\address{School of Physical Sciences, Jawaharlal Nehru University, New Delhi}
\email{vedgupta@mail.jnu.ac.in, ved.math@gmail.com}
\email{lavksingh@hotmail.com, lav17\_sps@jnu.ac.in} 



\date{}

\thanks{The second named author was  supported
  by the Council of Scientific and Industrial Research (Government
  of India) through a Senior Research Fellowship with
  No. \bf 09/263(1133)/2017-EMR-I}

\begin{abstract}
It was shown in \cite{Lav} that the Banach algebra $A:=S_2(\ell^2)\obp
S_2(\ell^2)$ is not Arens regular, where $S_2(\ell^2)$ denotes the
Banach algebra of the Hilbert-Schmidt operators on $\ell^2$. In this
article, employing the notion of limits along ultrafilters, we prove
that the irregularity of $S_2(\ell^2)\obp S_2(\ell^2)$ is not
strong. Along the way, we provide a class of functionals in $A^{**}$
which lie in the topological center but are not in $A$; and, as a
consequence, we deduce that $A^{**}$ is not an annihilator Banach
algebra with respect to any of the two Arens products.
\end{abstract}

\maketitle

\section{Introduction}
 For any Banach algebra $A$, Richard Arens (in \cite{Arens}) defined
two products $\Box$ and $\diamond$ on its bidual space $A^{**}$ such
that each product makes $A^{**}$ into a Banach algebra and the
canonical isometric inclusion $J:A\to A^{**}$ becomes a homomorphism
with respect to both the products. A Banach algebra $A$ is said to be
\textit{Arens regular} if the two products $\Box$ and $\diamond$ agree
on $A^{**}$, i.e. $f\Box g=f\diamond g$ for all $f,g\in A^{**}$. Soon,
people realized that when a Banach algebra is not Arens regular, then,
roughly speaking, it can exhibit different ``levels'' of
irregularity. Two such well explored notions are known as {\em strong
  Arens irregularity} (SAI) due to Dales and Lau (see \cite[Definition
  2.18]{DL}) and {\em extreme non-Arens regularity } (ENAR) due to
Granirer. In this article, we shall concentrate only on strong Arens
irregularity.

Briefly speaking, a Banach algebra $A$ is said to be {\em left}
(resp., {\em right}) {\em strongly Arens irregular} if its so called
{\em left topological center} $Z_t^{(l)} (A^{**})$ (resp., {\em right
  topological center} $Z_t^{(r)} (A^{**})$) equals $A$. And, $A$ is
said to be {\em strongly Arens irregular} if it is both left and right
strongly Arens irregular.  Interestingly, on the other extreme, it is
known that $A$ is Arens regular if and only if $Z_t^{(l)}
(A^{**}) = A^{**} = Z_t^{(r)} (A^{**})$ - see \cite[Page 1]{HN}.  It
thus follows that a Banach algebra is Arens regular as well as
strongly Arens irregular if and only if it is reflexive. In
particular, the Banach algebra $S_2(\mH)$ consisting of Hilbert-Schmidt
operators on a Hilbert space $\mH$ is Arens regular as well as
strongly Arens irregular.  However, in general, it is a difficult task
to realize whether an Arens irregular Banach algebra is strongly Arens
irregular or not. Yet, over the years, many known (non-reflexive)
Banach algebras have been identified which are strongly Arens
irregular and the quest is still on. For instance:
\begin{itemize}
  \item For a locally compact group $G$, it was proved by Lau and
    Losert \cite{Lau-losert} that $L^1(G)$ is always strongly Arens
    irregular.

\item The Fourier algebra $A(G)$ is strongly Arens irregular if $G$ is
  discrete and amenable - see \cite{Lau-Losert2}.
\item $A(\mathbb F_2)$ and $A(SU(3))$ are not strongly Arens
  irregular - see \cite{Losert}.
 
\end{itemize}

Our interest in the subject grew from the natural question of
analyzing the Arens regularity of the projective tensor product of
some standard Banach algebras. This project was, in fact, initiated in
the 80s by \"Ulger (see \cite{Ulger}) and over the years it has attracted
some good minds and has proved to be quite significant with a good
number of quality work already devoted to it - see, for instance,
\cite{Ulger, GI, Neufang, KR} and the references therein. One highlight
of this project has been a recent work of Neufang \cite{Neufang},
wherein he settled a four-decade-old open problem by proving that the
Varopoulos algebra $C(X) \obp C(Y)$, for (infinite) compact Hausdorff
spaces $X$ and $Y$, is neither Arens regular nor strongly Arens
irregular. In fact, Neufang goes on to prove that for arbitrary
$C^*$-algebras $A$ and $B$, their projective tensor product $A \obp B$
is Arens regular if and only if either $A$ or $B$ has the Phillips
property, thereby illustrating a deep relationship that exists between
Arens regularity of the projective tensor product and some intrinsic
geometric properties of the constituent Banach algebras.

An important class of Banach algebras consists of the Schatten
$p$-class operators, $S_p(\mH)$, on a Hilbert space $\mH$. They are
known to be Arens regular for all $1 \leq p < \infty$ (with respect to
multiplication as operator composition). Like the Varopoulos algebras,
a natural question that arises is to analyze the Arens regularity of
their projective tensor products $S_p(\mH) \obp S_q(\mH)$ for $1 \leq
p, q < \infty$. In \cite{Lav}, the second named author had shown that
the Banach algebras $S_p(\ell^2) \obp S_q(\ell^2)$ for $1 \leq p, q
\leq 2$ are not Arens regular. However, the question whether
$S_p(\ell^2) \obp S_q(\ell^2)$ is strongly Arens irregular or not
remained unanswered. Continuing the work initiated in \cite{Lav}, we
now answer this question in the negative for $p = q=2$. Thus,
analogous to the Varopoulos algebra, we now know that the Banach
algebra $S_2(\ell^2) \obp S_2(\ell^2)$ is neither Arens regular nor
strongly Arens irregular. Further, via some natural identifications,
essentially on similar lines, we observe that the predual
$S_1(\ell^2)$ of $B(\ell^2)$ inherits a Banach algebra structure which
is neither Arens regular nor strongly Arens irregular. (It must be
mentioned here that such structures on $S_1(\ell^2)$ are neither new
nor unique, as has been illustrated, for instance, in \Cref{predual}.)

Here is a brief overview of the flow of this article. After a quick
section on preliminaries, we first focus in Section $3$ on
establishing that the Banach algebra $A:=S_2(\ell^2)\obp S_2(\ell^2)$
is not strongly Arens irregular. The novelty of this paper lies in a
judicious exploitation of the notion of a limit along a non-principal
ultrafilter, which allows us to provide a class of functionals in
$A^{**}$ which lie in the topological center but are not in $A$;
thereby establishing that $A$ is not strongly Arens irregular. Also,
as another consequence of the technique of limit along a non-principal
ultrafilter, we deduce that the bidual space $A^{**}$ is not an
annihilator Banach algebra with respect to any of the two Arens
products.

\section{Preliminaries}
\subsection{Arens regularity and strong Arens irregularity}
        For the sake of convenience, we quickly recall the definitions
        of the two products $\Box$ and $\diamond$ mentioned in the
        Introduction.  Let $A$ be a Banach  algebra. For $a \in A$,
        $\omega\in A^*$, $f\in A^{**}$, consider the functionals
        $\omega_a, {}_{a}\omega\in A^*$, $\omega_f,{}_f\omega\in
        A^{**}$ given by $w_a=(L_a)^*(\omega)$,
        ${}_a\omega=(R_a)^*(\omega)$; $\omega_f(a)=f({}_a\omega)$ and
        ${}_f\omega(a)=f(\omega_a)$. Then, for $f,g\in A^{**}$ the
        operations $\Box$ and $\diamond$ are given by $(f\Box g)
        (\omega)= f({}_g\omega)$ and $(f\diamond
        g)(\omega)=g(\omega_f)$ for all $\omega \in A^*$. As recalled
        in the Introduction, $A$ is said to be Arens regular if the
        products $\Box$ and $ \diamond$ are same on $A^{**}$.

        Further, the left and the right topological centers of $A$ are defined,
respectively, as
\[
\begin{array}{rcl}
Z_t^{(l)} (A^{**}) & = & \{\varphi \in A^{**}: \phi\Box
\psi=\phi\diamond\psi \text{ for all } \psi\in A^{**} \}; \text{ and }
\\ Z_t^{(r)} (A^{**}) & = & \{\varphi \in A^{**}:
\psi\Box\phi=\psi\diamond \phi \text{ for all } \psi\in A^{**} \}.
\end{array}
\]
It (is known and) can be seen easily that $A \subseteq Z_t^{(l)}
(A^{**}) \cap Z_t^{(r)} (A^{**})$.

A Banach algebra $A$ is said to be {\em left} (resp., {\em right})
{\em strongly Arens irregular} if $Z_t^{(l)} (A^{**}) = A$ (resp.,
$Z_t^{(r)} (A^{**}) = A$). And, $A$ is said to be {\em strongly Arens
  irregular} if it is both left and right strongly Arens irregular. As
mentioned earlier, $A$ is Arens regular if and only if $Z_t^{(l)}
(A^{**}) = A^{**} = Z_t^{(r)} (A^{**})$. As a consequence of
Goldstine's Theorem, it is easy to see that any norm dense subset of
$A$ is weak*-dense in both $Z_t^{(r)}(A^{**})$ and
$Z_t^{(l)}(A^{**})$.  We refer the reader to \cite{DL} for further
discussion on topological centers.

\subsection{von Neumann Schatten classes}
For Hilbert spaces $\mathcal H_1$ and $\mathcal H_2$ and $1\leq
p<\infty$, the $p^{\mathrm{th}}$-Schatten class $S_p(\mathcal H_1,\mathcal
H_2)$ denotes the collection
\begin{equation}
  S_p(\mathcal H_1,\mathcal
  H_2):=\{T\in B(\mathcal H_1,\mathcal H_2)\, :\, 
  ||T||_p<\infty\},
\end{equation} where
  $||T||_p:=(\operatorname{Tr}(|T|^p))^{1/p}$ and $\text{Tr}$ denotes
the canonical semi-finite positive trace on $B(\mathcal{H}_1)$. It is
well known that $S_p(\mathcal H_1,\mathcal H_2)$ is a subspace of
$B_0(\mathcal H_1,\mathcal H_2)$, the space of compact operators from
$\mathcal{H}_1$ into $\mathcal{H}_2$, and is a Banach space with
respect to the Schatten $p$-norm $||\cdot||_p$.  Further, the finite rank operators are
dense in $\left(S_p(\mathcal H_1,\mathcal H_2), ||\cdot||_p\right)$.

 The following useful inequality is well known - see, for instance,
\cite[$\S XI.9.7$, Lemma 9(d)]{Dun-Schwartz}.

\begin{prop}\label{Th:2.1}
	Let $\mathcal H_1,\mathcal H_2$ and $\mathcal H_3$ be Hilbert
        spaces. If $R\in B(\mathcal H_2)$, $S\in S_p(\mathcal
        H_1,\mathcal H_2)$ and $T\in B( \mathcal H_1)$,
        then
        \[
        \|RST\|_p\leq \|R\| \|S\|_p \|T\|.
        \]
\end{prop}
\begin{rem}\label{schatten-facts}
  Let $\mH, \mH_i$, $i=1,2$ be Hilbert spaces.
  \begin{enumerate}
\item If $\mathcal{H}_1$ and $\mathcal{H}_2$ are separable, then the
 Schatten $p$-norm of an operator $T$ in $S_p(\mathcal H_1,\mathcal H_2)$ can
  also be expressed as $$\|T\|_p=\left(\sum_{i=1}^\infty
  s_i(T)^p\right)^{1/p},$$ where $\{s_i(T):i=1,2,\dots \}$ are the
  singular values of $T$, i.e., the square roots of the eigenvalues of
  $|T|$.
  \item  $\left(S_p(\mathcal{H}),
    \|\cdot\|_p \right)$ is a Banach $*$-algebra with respect to the
    composition of operators as multiplication.
\item $\left(S_2(\mathcal{H}), \|\cdot\|_2 \right)$ consists of
 the Hilbert-Schmidt operators on $\mathcal{H}$ and forms a Hilbert space with
  respect to the inner product $\la T, S\ra:=\mathrm{Tr}(S^*\circ T)$
  for $T, S \in S_2(\mathcal{H})$.
  \end{enumerate}
\end{rem}
  For more on Schatten classes, we refer the reader to \cite{Palmer, Dun-Schwartz}.

\subsection{Projective tensor product}
       Let $A$ and $B$ be Banach algebras. Then, there is a natural
       multiplication structure on their algebraic tensor product $A \otimes B$
       satisfying $(a_1\otimes b_1)(a_2\otimes b_2)=a_1a_2\otimes
       b_1b_2$ for all $a_i \in A$ and $b_i \in B$, $i=1,2$. And, there
       are various ways to impose a normed algebra structure on $A
       \otimes B$.  For instance, for each $u \in A\otimes B$, its
       projective norm is given by
        \[
       \Vert u \Vert_\gamma =\inf\left\{\sum_{i=1}^{n}\Vert a_i\Vert
       \Vert b_i\Vert~:~u=\sum_{i=1}^{n}a_i\otimes b_i\right\}.
       \]
       This norm turns out to be a \textit{cross norm}, i.e., $\|a
       \otimes b\|_{\gamma} = \|a\| \|b\|$ for all $(a,b) \in A \times
       B$; and, the completion of the normed algebra $A\otimes B$ with
       respect to this norm is a Banach algebra and is denoted by
       $A\otimes^\gamma B$.  As per our requirements related to the
       projective tensor product, we shall borrow  freely some results from \cite{Ryan}.

\subsection{Limits along ultrafilters}
We now provide a brief overview of the last (and the most important)
ingredient of this paper, namely, the notion of `` limits along 
  filters''. Recall that a {\em filter} on a set $X$ is a
collection $\mathcal{F} \subseteq \mathcal{P}(X)$ such that
\begin{enumerate}
\item $X \in \mathcal{F}$;
\item $\emptyset \notin \mathcal{F}$;
\item whenever $A \in \mathcal{F}$ and $A \subseteq B\subseteq X$, then $B \in
  \mathcal{F}$; and
  \item $A \cap B \in \mathcal{F}$ for every pair $A, B$ in
    $\mathcal{F}$.
\end{enumerate}
The {\em cofinite filter}  on $X$ is the collection of all cofinite subsets
of $X$. A filter $\mathcal{F}$ is said to be a {\em principal filter}
generated by an element $x$, if $\mathcal{F}=\{ A \subseteq X: x \in
A\}$.  And, a filter $\mathcal{F}$ on a set $X$ is said to be an {\em
  ultrafilter} if for any $A \subseteq X$, either $A$ or $A^c$ is in
$\mathcal{F}$. Clearly, every principal filter is an ultrafilter; and,
it is known, by Zorn's Lemma, that a non-principal ultrafilter exists
on every infinite set.

Further, given a filter $\mathcal{F}$ on a topological space $Y$, it
is said to converge to an element $y$ in $Y$ if every open set $U$
containing $y$ is in $\mathcal{F}$. And, given a filter $\mathcal{F}$
on a set $X$ and a map $f$ from $ X$ into a topological space $ Y$,
$f$ is said to converge along $\mathcal{F}$ to a point $y$ in $Y$ (and
denoted as $\lim_{\mathcal F}f(x) = y$) if the filter $f_*
\mathcal{F}:=\{ A\subseteq Y: f^{-1}(A) \in \mathcal{F}\}$ on $Y$
converges to $y$. Note that $f_*\mathcal{F}$ is an ultrafilter on $Y$
if $\mathcal{F}$ is so on $X$. The following well known elementary
observations are quite useful:
\begin{lemma} \label{ultrafilter-limit}
  \begin{enumerate}
\item Every non-principal ultrafilter on a set $X$ contains the
  cofinite filter on $X$.
    \item The limit of an ultrafilter on a Hausdorff space, if there
      exists one, is unique.
\item Every ultrafilter on a compact Hausdorff space converges to
  a unique point.
  \end{enumerate}
\end{lemma}
  
  \section{The projective tensor product of Hilbert-Schmidt space}
Throughout this and the following section, $\mathcal H $ will denote
an infinite dimensional separable Hilbert space, i.e., $\mH \cong
\ell^2$; and, $\mK$ will denote the Banach algebra $S_2(\mathcal H)$
consisting of the Hilbert-Schmidt operators on $\mH$.  Recall that,
given any orthonormal basis $\{e_i: i \in \N\}$ of $\mathcal H$, the
rank one operators $\{e_{ij}:=e_i\otimes e_j : i, j\in \N \}$ on $\mH$
given by $(e_i\otimes e_j)(\xi) =\la \xi, e_j\ra e_i$ for $\xi \in
\mH$, form an orthonormal basis of $\mK$. In particular, $\mK$ can be
identified with $\mH \overline{\otimes} \mH$ (Hilbert space tensor
product) as Hilbert spaces. Also, in the Banach algebra $S_2(\mH)$, we
have
  \begin{equation}\label{e-ij-relations}
e_{ij} e_{kl} = (e_i \otimes e_j) \circ (e_k \otimes e_l) = \delta_{j,
  k} e_i \otimes e_l = \delta_{j, k}e_{il}
 \end{equation}
 for all $i, j, k, l \in\N$.

\subsection{Some useful identifications}
  We shall frequently use some natural identifications that we mention
  in this subsection.

The identification in the following lemma is a direct adaptation of
\cite[Corollary 4.11]{Ryan} as per our requirements.
\begin{lemma}\label{obp-to-S1}
 There is a surjective conjugate-linear isometry $\theta:\mathcal
 K\obp\mathcal K \to S_1(\mathcal K,\mK^*)$ such that $\theta(x
 \otimes y)(z)=\left<z,x\right>\left<\cdot,y\right>$ for all $x,y, z
 \in \mK$.
\end{lemma}
\begin{proof}
For each $x , y \in \mK$, define $\varphi_{x,y}: \mK \rar \mK^*$ by
$\varphi_{x,y}(z) = \la x, z\ra \la \cdot, y\ra$ for $z\in
\mK$. Clearly, $\varphi_{x,y} \in S_1(\mK, \mK^*)$ and
$\|\varphi_{x,y}\|_1 = \|x\| \|y\| = \|x \ot y\|_{\gamma}$.

Note that, $\mK$ and $\mK^*$ being Hilbert spaces, it is easily
seen that the space of trace-class operators $S_1(\mK, \mK^*)$ is the
same as the space of nuclear operators $\mathcal{N}(\mK, \mK^*)$ and
that the trace-class norm is the same as the nuclear norm.  Consider the
conjugate Banach space $\overline{S_1(\mathcal{K},
  \mathcal{K}^*)}$. Then, on the lines of \cite[Corollary 4.11]{Ryan},
it can be shown that the linear mapping
\[
 \mK \ot \mK   \ni \sum_i x_i \ot
y_i \mapsto   \overline{\sum_i
  \varphi_{x_i, y_i}} \in \overline{S_1(\mK, \mK^*)}
\]
extends to a surjective linear isometric mapping
$\overline{\theta}:\mK \obp \mK\rar \overline{S_1(\mK, \mK^*)}$ such
that $\overline{\theta}(x\ot y)(z) = \overline{\varphi_{x,y}}(z)$ for all
$x,y, z \in \mK$. In particular, we obtain a surjective
conjugate-linear isometry $\theta:\mK \obp \mK \rar S_1(\mK, \mK^*)$
such that $\theta(x\ot y)(z) = \la z, x\ra \,\la \cdot, y\ra$ for all
$x,y, z \in \mK$.
\end{proof}
\begin{rem}\label{natural-identifications}
\begin{enumerate}
 \item In view of \Cref{obp-to-S1}, to each $u\in \mathcal K\obp\mathcal K$, we
   can associate an infinite matrix $[u_{ij, kl}]_{i,j,k,l\in \mathbb
     N}$, namely, the matrix of $\theta(u)$ with respect to the
   orthonormal basis $\{e_{ij} :i, j \in \N\}$ of $\mK$. Thus, $u_{ij,
     kl}:=\theta(u)(e_{kl})(e_{ij})$ for all $i, j, k, l
   \in \N$.\smallskip

Whenever convenient, we shall interchangeably use the
matrix  $u=\left[u_{ij, kl}\right]$ and
the element $u \in \mathcal K\obp\mathcal K$.

\item  There is natural
  surjective isometry
  \[
  \vartheta: B(\mathcal K,\mathcal K^*) \to
  (\mathcal{K}\obp\mathcal K)^*
  \]
  satisfying
	\[
	 \vartheta(f)\left(x \otimes y\right)= f(x)(y)
         \]
         for all $f \in B(\mK,\mK^*)$, $x,y\in \mK$ - see \cite[pg 24]{Ryan}.
\end{enumerate}
          (Notice that these identifications are not  algebra isomorphisms.) 

\end{rem}
 
\begin{lemma}\label{left-product}
With notations as in the preceding remark, for each $u=[u_{ij,
    kl}]_{i,j,k,l\in \mathbb N}$ in $\mathcal K\obp\mathcal K$ and
$m, n \in \N$, the matrix corresponding to the product $(
e_{mn}\otimes e_{mn})\cdot u$ is given by $[v_{ij, kl}]_{i,j,k,l\in
  \mathbb N}$, where $v_{ij, kl}=0$ if $i\neq m$ or $k\neq m$, and
$v_{mj, ml}=u_{nj, nl}$ for all $l,j\in \N$.
\end{lemma}

\begin{proof}
	Recall, from \cite[$\S 2.1$]{Ryan}, that $u$ can be expressed
        as a sum $u=\sum_{i=1}^\infty \beta_i x_i\otimes
        y_i$ for some pair of null sequences $\{x_i\}$ and
        $\{y_i\}$ in  $\mathcal K$,
        and a sequence of scalars $\{\beta_i\}$ satisfying
        $\sum_{i=1}^\infty|\beta_i|<\infty$. Then, we have
        \begin{align*}
          u_{ij,kl}&=\theta(u)(e_{kl})(e_{ij})\\&=\sum_{r=1}^\infty\beta_r
           \la e_{kl}, x_r\ra \left<e_{ij},y_r\right>&&
          (\text{by  \Cref{natural-identifications}}(1))
	\end{align*}
        for all $i, j, k, l \in \N$.  Since, $(
        e_{mn}\otimes e_{mn})\cdot u=\sum_{i=1}^\infty \beta_i (e_{mn} x_i)\otimes (e_{mn}y_i)$, we obtain
        \[
        v_{ij,kl} = \theta\big(( e_{mn}\otimes e_{mn})\cdot
        u\big) (e_{kl})( e_{ij})=
        \sum_{r=1}^\infty \beta_r \la e_{kl}, e_{mn}
          x_r\ra\la  e_{ij},e_{mn} y_r\ra
          \]
                  for all $i, j, k, l \in \N$.  Notice that, by
                  \Cref{e-ij-relations}, $\la e_{kl}, e_{mn} x_r\ra=
                  \left<e_{nm} e_{kl}~, x_r\right> = 0$ if $k\neq m$
                  and $\la  e_{ij},e_{mn} y_r\ra = \left< e_{nm}
                  e_{ij},y_r \right> = 0$ if $i\neq m$.  Hence,
                  $v_{ij,kl}=0$ if either $i\neq m$ or $k\neq
                  m$. Further, it can be easily checked from the above
                  equations that
        \[
        v_{mj, ml}=\sum_{r=1}^\infty\beta_r
        \la e_{nl}, x_r\ra \la  e_{nj},y_r\ra = u_{nj, nl}
        \]
        for all $j, l \in \N$.
\end{proof}

\subsection{Functionals induced by non-principal ultrafilters} The following observation is the crux of this section.
\begin{prop}\label{phi-U}
  Let $A:= \mK \obp \mK$, $\mathcal U$ be a non-principal
  ultrafilter on $\mathbb N$ and $J:A \to A^{**}$ denote the canonical isometry. Then, the following hold:
\begin{enumerate}
\item The sequence $\{J(e_{nn} \ot e_{nn})\}$ converges along
  the ultrafilter $\mathcal{U}$, with respect to the weak*-topology, to a unique element
  (denoted as) $\phi_{\mathcal{U}} \in A^{**} \setminus J(A)$, which satisfies
\[
\phi_{\mathcal
  U}(f)=\lim_{\mathcal U}\left<f(e_{nn}),e_{nn}\right> \text{ for all $f \in A^{*}$}.
\]

\item For each $f \in A^*$, the function
\[
\N
\ni n \mapsto f_{ e_{nn}\otimes e_{nn}} \in A^*
\]
converges along the ultrafilter $\mU$ to a unique element in the
$w^*$-compact set $\{g \in A^*: \|g\| \leq \|f\|\}$ in $A^*$ with
respect to the $w^*$-topology.
\item We have
\begin{equation} \label{eqn:3.1}
(\phi_{\mathcal U}\diamond
  \psi)(f)=\psi\left(\overset{w^*}{\lim_{\mathcal U}} f_{ e_{nn}\otimes
    e_{nn}}\right) \mathrm{ and }\ (\phi_{\mathcal U}\Box \psi)(f) =
  \lim_{\mathcal U}\psi(f_{ e_{nn} \otimes e_{nn}})
\end{equation}
for all $\psi \in A^{**}$  and  $f\in A^*$.
\end{enumerate}
\end{prop}

\begin{proof}
  (1): Clearly, the sequence $\{J(e_{nn} \ot e_{nn})\}$ is contained
  in the closed unit ball $B_1(A^{**})$, a $w^*$-compact set. Hence,
  by \Cref{ultrafilter-limit}, it converges to a unique element along
  the ultrafilter $\mathcal{U}$, say, $\phi_{\mathcal{U}}$ in
  $B_1(A^{**})$, with respect to the weak*-topology. Also, we clearly
  see, via \Cref{natural-identifications}(2), that
  \[
\phi_{\mathcal U}(f) = \lim_{\mU}\ J(e_{nn}\ot e_{nn})(f)=
\lim_{\mU}\la f(e_{nn}), e_{nn}\ra \text{ for all } f \in A^*.
\]
Notice that $\phi_{\mathcal U}\neq0$ because $\phi_{\mathcal
  U}(\eta)=1$, where $\eta:\mathcal K\to\mathcal K^*$ is defined as
$\eta(e_{ij})=\left<\cdot,e_{ij}\right>$ for each $i,j\in \mathbb
N$. Further, since $(\mK\obp \mK)^{**}\cong B(\mK,\mK^*)^*$, in order
to show that $\phi_{\mathcal U}\notin J({\mathcal K\obp\mathcal K})$,
it suffices to show that $\phi_{\mathcal U}$ vanishes on the space of
finite rank operators $B_{00}(\mK,\mK^*)$ whereas $J(u)$ does not
vanish on $B_{00}(\mK,\mK^*)$ for any non-zero $u \in\mK \obp \mK$.

To see this, let $f:\mK\to\mK^*$ be a finite rank operator defined as
$f(z) =\sum_{i=1}^k \lambda_i\left<z ,x_i\right>\la\cdot, x_i\ra$,
where $\{x_i\}_{i=1}^k$ is an orthonormal set in $\mathcal K$ and
$\{\lambda_i\}_{i=1}^k$ are some scalars. Note that, $e_{nn} \to
0$ weakly in $\mK$; so, the function $\N \ni n \mapsto \la x_i, e_{nn}
\ra \in \C$ converges to $0$ along the non-principal ultrafilter $\mU$
(as $\mU$ contains the cofinite filter on $\N$).  Thus, via the
identification $B(\mK,\mK^*)^* \cong (\mK \obp \mK)^{**}$, we obtain
\[
\phi_{\mathcal U}(f)  = \lim_{\mathcal
  U}\sum_{i=1}^k\lambda_i|\left<x_i,e_{nn}\right>|^2=0.
\]
In particular,  $\phi_\mU$ vanishes on $B_{00}(\mK)$.

Let $0\neq u \in \mK \obp \mK$. Then, the (possibly finite) sequence
$\{\mu_i\}$ consisting of the singular values of $\theta(u) \in
S_1(\mK,\mK^*)$ is absolutely summable. Also, by the singular value
decomposition of $\theta(u)$, there exist (possibly finite) orthonormal sequences
$\{x_i\}$ and $\{y_i\}$ in $\mK$ such that $\theta(u)(x)= \sum_i \mu_i
\la x , x_i\ra \la\cdot,y_i\ra$ for all $x \in \mK$, i.e., the
sequence of operators
\[
\mathcal{K} \ni z\mapsto \sum_{i=1}^n \mu_i\la z, x_i\ra \otimes
\la\cdot,y_i\ra \in \mathcal{K}^*
\]
converges strongly to $\theta(u)$. In fact, since $\sum_i \mu_i
< \infty$ and since $\|\cdot \|_\gamma$ is a cross-norm, it follows
that $u = \sum_{i=1}^{\infty} \mu_i x_i \otimes y_i$ in $\mK \obp
\mK$. Then, we observe that $J(u)(x_i \otimes y_i) = \mu_i \neq 0$ for
all $i$. Hence, $J(u)$ does not vanish on $B_{00}(\mK)$.
\smallskip

(2): Let $f \in A^*$. Then,  we have
\[
\|f_{e_{nn} \otimes e_{nn}}\| \leq \|f\|\|
e_{nn}\otimes e_{nn}\|_{\gamma} \leq \|f\|
\]
for all $n\in \N$.  Since the closed unit ball of $A^*$ is
weak*-compact, it follows from \Cref{ultrafilter-limit} that the
function $\N\ni n \mapsto f_{e_{nn}\otimes e_{nn}}$ converges along
the ultrafilter $\mU$ to a unique element in the weak*-compact set
$\{g \in A^*: \|g\| \leq \|f\|\}$, with respect to the weak*-topology
on $A^*$.\smallskip

(3): Let $\phi, \psi \in A^{**}$ and $f\in A^*$.  Notice
that, due to double limit criteria $$\phi\diamond
\psi=\lim_{\beta}\lim_{\alpha}f(\phi_\alpha\psi_\beta)=\lim_\beta
\left(\overset{w^*}{\lim_\alpha}
f_{\phi_\alpha}\right)(\psi_\beta)=\psi(\overset{w^*}{\lim_\alpha}
f_{\phi_\alpha})$$ for any pair of nets $\{J(\phi_\alpha)\}$ and
$\{J(\psi_\beta)\}$ in $J(A)$ converging to $\phi$ and $\psi$, respectively, in the
weak* topology of $A^{**}$. Since $\phi_{\mathcal
  U}=\overset{w^*}{\lim}_{\mathcal U} J(e_{nn}\otimes e_{nn})$, using 
(2), we deduce that
\[
(\phi_{\mathcal U}\diamond
\psi)(f)=\psi\left(\overset{w^*}{\lim_{\mathcal U}}f_{ e_{nn}\otimes
  e_{nn}}\right).
\]

And, for the box product, we have
\begin{eqnarray*}
(\phi_{\mathcal U}\Box \psi)(f)& =&\phi_{\mathcal U}(_\psi f)\\ &
  =&\lim_{\mathcal U}\left<_\psi f(e_{nn}),~
  e_{nn}\right>\\ &=&\lim_{\mathcal U}~_\psi f(e_{nn}\otimes e_{nn})
  \qquad \qquad \text{(by \Cref{natural-identifications}(2))}\\ & =&
  \lim_{\mathcal U}\psi(f_{ e_{nn} \otimes e_{nn}}).
\end{eqnarray*}

\end{proof}
\subsection{The main theorem}
We are now all set to establish the main result of this article that
the Banach algebra $A:=\mK \obp \mK$ is not strongly Arens
irregular. We first analyze the annihilating properties of its bidual
space $(\mK \obp \mK)^{**}$ with respect to the two Arens products
$\Box$ and $\diamond$.

Recall that, for any subset $S$ of a Banach algebra $B$, the {\em left} and {\em right
annihilator ideals} of $S$ in $B$ are defined, respectively, as
\[\begin{array}{rcl}
\mathcal{A}^l_B(S)& = & \{x\in
  B~:~xs=0 \text{ for all } s\in
  S\}; \text{ and}\\ \mathcal{A}^r_B(S)&= &\{x\in
  B~:~sx=0 \text{ for all } s\in S\}.
\end{array}
\]
$B$ is said to be an {\em annihilator Banach algebra} if
$\mathcal{A}^l_B(B)= (0) = \mathcal{A}^r_B(B)$ and for every proper
closed left (resp., right) ideal $I$ (resp., $J$), $\mathcal{A}^r_B(I)
\neq (0 ) \neq \mathcal{A}^l_B(J)$. Most of the common Banach algebras
(such as function spaces and operator algebras) are annihilator
algebras. For more, see \cite{Palmer, Rickart}.\smallskip

An element $T \in S_1(\mK,\mK^*)$ is said to have finite support if its
matrix $[T]_{ij, kl}$ (as in \Cref{natural-identifications}(1)) has only finitely many non-zero entries.
\begin{thm}\label{ann-diamond}
	For any non-principal ultrafilter $\mathcal U$ on $\mathbb
        N $,
        $$
        \mathcal{A}^r_{(A^{**},\diamond)}\left(\{\phi_{\mathcal
          U}\}\right)=A^{**}.
        $$
\end{thm}

\begin{proof}
Let $\psi \in A^{**}$ and $f \in A^*$. We first assert that
        $f_{ e_{nn}\otimes e_{nn}}\stackrel{w^*}{\longrightarrow} 0$ as $n \to \infty$.\smallskip

 To prove this, it is sufficient to show that $\lim_{n\to \infty}f_{
   e_{nn}\otimes e_{nn}}(u) = 0$ for all $u \in \mK \obp \mK$. Notice
 that $(e_{nn}\otimes e_{nn})\cdot u\to 0$ in norm for all $u\in
 \mathcal K\otimes\mathcal K$ due to \Cref{left-product}. Hence, it
 converges to $0$ in norm for all $u\in \mathcal K\obp\mathcal K$,
 because $\mathcal K\otimes \mathcal K$ is dense in $\mathcal K
 \obp\mathcal K$. Thus, $f_{e_{nn}\otimes e_{nn}}(u)=f((e_{nn}\otimes
 e_{nn})\cdot u)\to 0$ for all $u\in \mathcal K\obp\mathcal K$; thereby, implying that
 $f_{e_{nn}\otimes e_{nn}}\stackrel{w^*}{\longrightarrow}0$.

Next, since $\mathcal U$ contains the cofinite filter - see
\Cref{ultrafilter-limit}, the set $\{n : f_{e_{nn}\otimes e_{nn}}\in
U\}$, being cofinite, is in $\mathcal U$, for each weak* open
neighborhood $U$ of $0$. Hence, $0$ is a weak* limit of the function
$\N \ni n \mapsto f_{e_{nn}\otimes e_{nn}}\in A^*$ along the
ultrafilter $\mathcal U$.  And, since the limit along an ultrafilter
in a Hausdorff space is unique, it follows from \Cref{eqn:3.1} that
$(\phi_{\mathcal U}\diamond \psi)(f)=0$.  Since $\psi$ and $f$ were
arbitrary, this implies that
	$$\mathcal{A}^r_{(A^{**},\diamond)}(\{\phi_{\mathcal
  U}\})=A^{**}.$$
\end{proof}

Now we turn to the first Arens product $\Box$ on $(\mK \obp
\mK)^{**}$.
\begin{thm}\label{ann-box}
	For any non-principal ultrafilter $\mathcal U$ on $\mathbb
        N$, $$\mathcal{A}^r_{(A^{**},\Box)}\left(\{\phi_{\mathcal
          U}\}\right)=A^{**}.$$
	\end{thm}
\begin{proof}
Let $\psi\in (\mK \obp \mK)^{**}$ and $f\in (\mathcal K\obp\mathcal
K)^{*}$. Then, from \Cref{eqn:3.1}, we know that $(\phi_{\mathcal U}
\Box \psi)(f) = \lim_{\mathcal{ U}} \psi(f_{e_{nn} \ot e_{nn}})$.
Since $\mathcal{U}$ contains the co-finite ultrafilter, it can be
easily seen that $\lim_{\mathcal{ U}} \psi(f_{e_{nn} \ot e_{nn}})$ equals
$\lim_{n\to \infty} \psi(f_{e_{nn} \ot e_{nn}})$ if the latter exists. So, in
order to show that $ (\phi_{\mathcal U}\Box \psi)(f)=0 $, it suffices
to show that $\lim_{n\to \infty}\psi(f_{ e_{nn} \otimes e_{nn}}) =0$.

To prove this, let us assume on the contrary that the sequence
$\{\psi(f_{e_{nn} \otimes e_{nn}})\}$ does not converge to $0$ as $n$
tends to infinity. Then, there exists an $\epsilon>0$ for which we can choose an
increasing sequence $\{n_t\}_{t=1}^\infty$ of natural numbers such
that $|\psi(f_{ e_{n_tn_t} \otimes e_{n_tn_t}})|>\epsilon$ for all $t
\geq 1$.  Let us define $g_{t}\in A^*$ for each natural number $t$ as
$g_{t}=c_{t}\cdot f_{ e_{n_tn_t}\otimes e_{n_tn_t}},$ where
$c_t=\frac{\overline{\psi(f_{ e_{n_tn_t} \otimes
      e_{n_tn_t}})}}{\left\vert\psi(f_{e_{n_tn_t} \otimes
    e_{n_tn_t}})\right\vert}$. Now, for each $N \in \N$, let
$h_N:=\sum_{t=1}^Ng_{t}$. Then,
\[
	|h_N(u)| = \left\vert f\left (\sum_{t=1}^N c_t(
        e_{n_tn_t}\otimes e_{n_tn_t})\cdot u\right)\right\vert \leq
        ||f||\,
        \left\vert\left\vert\sum_{t=1}^N(c_tu_{n_t})\right\vert\right\vert_{\gamma}
\]
for all $u \in \mK\obp \mK$ and $N \in \N$, where $u_{n_t}=(e_{n_tn_t}\otimes e_{n_tn_t})\cdot u$.

For each $N \in \N$, consider the linear operator $Q_N:\mathcal K\to
\mathcal K$ satisfying $Q_N(e_{n_tj})=c_t e_{n_tj}$ for $t=1,2,..,N$,
$j\in \mathbb N$ and $Q_N(e_{ij})=0$ otherwise. Clearly,
$||Q_N||=1$. Then, one can easily verify, through the actions on the
orthonormal basis $\{e_{ij}\}$, that
\[
\theta\left(\sum_{t=1}^Nc_tu_{n_t}\right)=Q_N^*\theta(u)Q_N
\]
for all $u \in \mK \obp \mK$ and $N \in \N$. Thus, from \Cref{Th:2.1},
we obtain
\[
\left\|\theta\left(\sum_{t=1}^Nc_tu_{n_t}\right)\right\|_1=\left\|\sum_{t=1}^Nc_tu_{n_t}\right\|_\gamma\leq
||u||_\gamma
\]
for all $u \in \mK \obp \mK$ and $N \in \N$. In particular,
\[
  |h_N(u)|\leq |f||||u||_\gamma
  \]
  for all $u \in \mK\obp \mK$ and $N \in \N$.  Thus, $||h_N||\leq
  ||f||$ for each natural number $N$.

  On the other hand,
  \[
  \psi(h_N)=\sum_{t=1}^N\psi(g_t)=\sum_{t=1}^N|\psi(f_{
    e_{n_tn_t}\otimes e_{n_tn_t}})| \text{ for all } N \in \N.
  \]
  Thus, $\{\psi(h_N): N \in \N\}$ is an unbounded sequence. But this
  is absurd because
  \[
  |\psi(h_N)|\leq \|\psi \|\|h_N\| \leq \|\psi\| \|f\|
  \]
  for all $N \in \N$. Hence, our assumption was wrong and we must have $\lim_{n\to
    \infty}\psi(f_{ e_{nn}\otimes e_{nn}})=0$.

  Since $\psi$ and $f$ were arbitrary, it follows that
  $\mathcal{A}^r_{(A^{**},\Box)}\left(\{\phi_{\mathcal
    U}\}\right)=A^{**}.$
	\end{proof}

It is worth noting that $\mathcal K\obp\mathcal K$ does not posses
annihilating elements. However, we can deduce from \Cref{ann-diamond}
and \Cref{ann-box} that its bidual behaves differently with respect to
both Arens products.

\begin{cor}	$(\mathcal K\obp\mathcal K)^{**}$ is not an annihilator
  Banach algebra with respect to any of the two Arens products.
\end{cor}

\begin{rem}
	From the preceding discussions, we observe that
        \[
        \mathcal{A}^r_{(A^{**},\diamond)}(\{\phi_{\mathcal
          U}\})=A^{**} =
        \mathcal{A}^r_{(A^{**},\Box)}(\{\phi_{\mathcal U}\}).
        \]
        This is not very common to see in usual Banach algebras. For
        example, $C^*$-algebras, group algebras $L^1(G)$, Schatten
        spaces $S_p(\mH)$ and many other standard Banach algebras do
        not possess elements which can annihilate the whole algebra.
\end{rem}

\begin{thm} \label{JA-plus-phi-U}
The Banach algebra $A:=\mathcal K\obp\mathcal K$ is neither Arens regular nor 
strongly Arens irregular.
  \end{thm}
\begin{proof}
  That $A$ is not Arens regular is known from \cite{Lav}. To exhibit the failure of strong irregularity, we  assert that
  \[
J(A)+\{\phi_{\mathcal U} : \mathcal U \text{ is a non-principal
  ultrafilter on }\mathbb N \}\subseteq Z_t^{(l)}(A^{**}).
\]
For any non-principal ultrafilter $\mathcal{U}$ on $\N$, it follows from \Cref{ann-diamond} and \Cref{ann-box} that
  \[
  (\phi_{\mathcal U}+J_a)\diamond \psi =\phi_{\mathcal
    U}\diamond\psi+J_a\diamond \psi = 0 + J_a\Box \psi
  =(\phi_{\mathcal U}+J_a)\Box \psi
  \] for all $a \in A$ and $\psi\in A^{**}$.
  Since non-principal ultrafilters exist on every infinite set, it
  follows that the left topological center of $A$ contains more than
  $J(A)$; so, $A$ is not left strongly Arens irregular. In particular,
  $A$ is not strongly Arens irregular.
\end{proof}

\begin{rem}
Computations as in \Cref{phi-U} and \Cref{ann-box} yield
\[
(\psi\diamond\phi_{\mathcal U})(f)= \lim_{\mathcal U}\psi\left(~_{
  e_{nn}\otimes e_{nn}}f\right) ~~\text {and} \]
$$(\psi\Box\phi_{\mathcal
  U})(f)=\psi\left(\overset{w^*}{\lim_{\mathcal U}}~_{e_{nn}\otimes
  e_{nn}}f\right)
 $$ for all $\psi \in A^{**}$ and $f \in A^*$. Now, applying
symmetrical arguments as in \Cref{ann-box} and \Cref{ann-diamond}, we
can easily conclude that $\phi_{\mathcal U}$ is in the left
annihilator ideal of $A$ and hence in the right topological
center. Hence, $A$ is not right strongly Arens irregular.
\end{rem}

\subsubsection{The predual of $B(\ell^2)$} 
There is a natural Banach algebra
structure on the dual space $\mathcal K^*$, where the multiplication
is given by
\[
\la \cdot , T \ra \, \la \cdot , S \ra: = \la\cdot , T \circ S
\ra\text { for } T, S \in \mK.
\]
Recall that, upto isometric isomorphism, the
predual of the von Neumann algebra $B(\mH)$ is given by the space
$S_1(\mH)$ consisting of the trace-class operators on $\mH$; and, it
is known that $S_1(\mH)$ is an Arens regular Banach algebra with
respect to the multiplication given by the usual composition of
operators.

 On the other hand, there exists a natural isometric identification
 between $S_1(\mH)$ and the projective tensor product $\mH^* \obp
 \mH$. Since $\mH \cong \mK$ as Hilbert spaces, from the fact that
 $\mK$ is a Banach algebra with respect to the operator composition
 and the Schatten $2$-norm (\Cref{schatten-facts}), it follows that
 $S_1(\mH)$ inherits a canonical Banach algebra structure from that of
 $\mK^* \obp \mK$. Using the isometric isomorphism of $\mathcal
 K^*\obp\mathcal K$ with $\mK\obp\mK$, one can conclude that the
 pre-dual of $B(\mathcal H)$ is neither Arens regular nor strongly
 irregular with respect to this new induced multiplication.

\begin{rem}\label{predual}
  \begin{enumerate}
\item  The preceding theorem gives a Banach algebra structure on the
  predual $S_1(\ell^2)$ of $B(\ell^2)$, which is neither Arens regular
  nor strongly Arens irregular. This is in contrast to the fact that
  $S_1(\ell^2)$ with usual operator compostion is Arens regular.
\item However, such structures on $S_1(\ell^2)$ are already known to
  the experts and are not unique. An anonymous expert brought the
  following example to our notice:

  We know that $\ell^1$ sits as
a complemented subspace (known as tensor diagonal) in
$X:=\ell^2\obp\ell^2$. Thus, $X=\ell^1\oplus Y$ for some closed space
$Y$. We now define a product $\bullet$ on $X$ as $(a+y)\bullet
(b+z)=a\star b$ for $a,b\in \ell^1$ and $y,z\in Y$, where $\star$ is
the convolution on the semigroup Banach agebra $\ell^1=
\ell^1(\N)$. Clearly, $\bullet$ is well defined and $X$ is a
commutative Banach algebra with respect to this new product.

Note that, $X$ is not Arens regular, because $\ell^1$ is a non-regular
subalgebra of $X$. And, since $Y^{**}\hookrightarrow X^{**}$ and $X
\bullet Y = \{0\}$, we observe that $X$ fails to be strongly Arens
irregular as well.
\end{enumerate}
  \end{rem}

\subsection*{Some questions}
We conclude by listing a few unresolved natural questions:
\begin{enumerate}
\item Is a complete characterization of the topological centers of
  $\mK \obp \mK$ (or $\mK^*\obp\mK$) possible?
\item Is $\mathcal K^* \obp \mathcal K$ (equivalently, $\mK \obp \mK$)
  extremely non-Arens regular?
  \item In \cite{Lav}, it was shown that $S_p(\ell^2) \obp S_q(\ell^2)$
    is not Arens regular for all $1 \leq p, q \leq 2$. So, it still
    remains to answer whether $S_p(\ell^2) \obp S_q(\ell^2)$ is
    strongly Arens irregular for each pair $1 \leq p, q \leq 2$
    with $(p,q) \neq (2,2)$.
\item Does the predual of an arbitrary von Neumann algebra admit any
  natural Banach algebra structure? If yes, then what can be said about its
  Arens regularity?
\end{enumerate}

\noindent{\bf Acknowledgements.} The authors would like to the thank
the anonymous referee for suggesting some simplifications of some of
the proofs.

\end{document}